\documentclass[11pt]{article}

\usepackage{amsmath, amsthm}
\usepackage{amsmath, amsfonts}
\usepackage{amsmath, amssymb}
\usepackage{amsmath}
\usepackage[all]{xy}

\newtheorem{definition-lemma}[theorem]{Definition/Lemma}
\newtheorem{definition-explanation}[theorem]{Definition/Explanation}
\newtheorem{explanation-definition}[theorem]{Explanation/Definition}
\newtheorem{definition-fact}[theorem]{Definition/Fact}
\newtheorem{definition-notation}[theorem]{Definition/Notation}

\newtheorem{lemma-definition}[theorem]{Lemma/Definition}

\newtheorem{remark-notation}[theorem]{\it Remark/Notation}

\newtheorem{example-definition}[theorem]{Example/Definition}

\newtheorem{definition-prototype}[theorem]{Definition-Prototype}

\numberwithin{equation}{subsection}

\newtheorem{stheorem}{Theorem}[section]
\newtheorem{sdefinition}[stheorem]{Definition}
\newtheorem{sdefinition-lemma}[stheorem]{Definition/Lemma}
\newtheorem{sdefinition-explanation}[stheorem]{Definition/Explanation}
\newtheorem{sexplanation-definition}[stheorem]{Explanation/Definition}
\newtheorem{sdefinition-fact}[stheorem]{Definition/Fact}
\newtheorem{sdefinition-notation}[theorem]{Definition/Notation}
\newtheorem{slemma}[stheorem]{Lemma}
\newtheorem{slemma-definition}[stheorem]{Lemma/Definition}
\newtheorem{sproposition}[stheorem]{Proposition}
\newtheorem{scorollary}[stheorem]{Corollary}
\newtheorem{sremark}[stheorem]{\it Remark}
\newtheorem{sremark-notation}[stheorem]{\it Remark/Notation}
\newtheorem{sconjecture}[stheorem]{Conjecture}

\newtheorem{sexample-definition}[stheorem]{Example/Definition}

\newtheorem{sdefinition-prototype}[stheorem]{Definition-Prototype}

\newtheorem{ssdefinition-lemma}[sstheorem]{Definition/Lemma}
\newtheorem{ssdefinition-explanation}[sstheorem]{Definition/Explanation}
\newtheorem{ssexplanation-definition}[sstheorem]{Explanation/Definition}
\newtheorem{ssdefinition-fact}[stheorem]{Definition/Fact}
\newtheorem{ssdefinition-notation}[theorem]{Definition/Notation}

\newtheorem{sslemma-definition}[sstheorem]{Lemma/Definition}

\newtheorem{ssremark-notation}[sstheorem]{\it Remark/Notation}

\newtheorem{ssexample-definition}[sstheorem]{Example/Definition}

\newtheorem{ssdefinition-prototype}[sstheorem]{Definition-Prototype}

\newcommand{\Ann}{\mbox{\it Ann}\,}
\newcommand{\CohCategory}{\mbox{\it ${\cal C}$\!oh}\,}
\newcommand{\Endsheaf}{\mbox{\it ${\cal E}\!$nd}\,}
\newcommand{\Ext}{\mbox{\rm Ext}\,}
\newcommand{\Span}{\mbox{\it Span}\,}
\newcommand{\Spec}{\mbox{\it Spec}\,}
 \newcommand{\boldSpec}{\mbox{\it\bf Spec}\,}
\newcommand{\Supp}{\mbox{\it Supp}\,}

\newcommand{\bii}{\mbox{\it b.i.i.}}
\newcommand{\id}{\mbox{\it id}\,}
\newcommand{\lcm}{\mbox{\it l.c.m.}}
\newcommand{\length}{\mbox{\it length}\,}
\newcommand{\pr}{\mbox{\it pr}}

\begin{document}

\enlargethispage{24cm}

\begin{titlepage}

$ $

\vspace{-1.5cm} 

\noindent\hspace{-1cm}
\parbox{6cm}{\small November 2011}\
   \hspace{7cm}\
   \parbox[t]{5cm}{yymm.nnnn [math.AG] \\
                   D0 resolution of singularity\\
                   D(9.1): curve}

\vspace{2cm}

\centerline{\large\bf
 D0-brane realizations of the resolution of a reduced singular curve}


\vspace{3em}

\centerline{\large
  Chien-Hao Liu
   \hspace{1ex} and \hspace{1ex}
  Shing-Tung Yau
}

\vspace{6em}

\begin{quotation}
\centerline{\bf Abstract}

\vspace{0.3cm}

\baselineskip 12pt  
{\small
 Based on examples from superstring/D-brane theory
  since the work of Douglas and Moore on resolution of singularities
   of a superstring target-space $Y$ via a D-brane probe,
 the richness and the complexity of the stack of punctual D0-branes
  on a variety, and
 as a guiding question, we lay down a conjecture
  that any resolution $Y^{\prime}\rightarrow Y$
   of a variety $Y$ over ${\Bbb C}$
  can be factored through an embedding of $Y^{\prime}$
  into the stack ${\frak M}^{0^{A\!z^f}_{\;p}}_r\!\!(Y)$
  of punctual D$0$-branes of rank $r$ on $Y$ for $r\ge r_0$ in ${\Bbb N}$,
  where $r_0$ depends on the germ of singularities of $Y$.
 We prove that this conjecture holds
  for the resolution $\rho: C^{\prime}\rightarrow C$
   of a reduced singular curve $C$ over ${\Bbb C}$.
 In string-theoretical language,
  this says that
 the resolution $C^{\prime}$ of a singular curve $C$ always arises
  from an appropriate D0-brane aggregation on $C$ and  that
 the rank of the Chan-Paton module of the D0-branes involved
  can be chosen to be arbitrarily large.
} 
\end{quotation}

\vspace{13em}

\baselineskip 12pt
{\footnotesize
\noindent
{\bf Key words:} \parbox[t]{14cm}{D-brane, resolution, singularity;
  punctual D0-brane, stack; singular curve, normalization;\\
  embedding, separation of points, separation of tangents.
 }} 

\bigskip

\noindent {\small MSC number 2010: 14E15, 14A22, 81T30.
} 

\bigskip

\baselineskip 10pt
{\scriptsize
\noindent{\bf Acknowledgements.}
In fall 2011,
 Baosen Wu
  gives a topic course at Harvard on moduli of coherent sheaves.
This note arises from a series of originally-casual-but-later-serious
 discussions with him.
Despite his generosity, insisting that he plays only a mild role in this,
we regard him as a coauthor and
 thank him for serving as a sounding board to our preliminary thought,
 catching an incompleteness of the first draft of the note, and
 influencing our mathematical understanding of B-branes and other issues.
C.-H.L.\
 thanks in addition
 Carl Mautner
  for discussions on constructible/perverse sheaves
  that influence his thought on A-branes;
 Sean Keel
  for an exceptional lecture on mirror symmetry
  and answers to his various questions;
 Clay Cordova, Babak Haghighat, Cumrun Vafa
  for discussions on new issues on A-branes;
 Lara Anderson, Feng-Li Lin, Li-Sheng Tseng for conversations
  that enrich his stringy culture;
 Nir Avni, Jacob Lurie, C.M., B.W.\ for topic courses, fall 2011;
 and Ling-Miao Chou for moral support.
S.-T.Y.\ thanks in addition
 Department of Mathematics at National Taiwan University
 for the intellectually rich environments and hospitality.
The project is supported by NSF grants DMS-9803347 and DMS-0074329.
} 

\end{titlepage}

\newpage

\begin{titlepage}

$ $

\vspace{12em} 

\centerline{\small\it
 Chien-Hao Liu dedicates this note to}
\centerline{\small\it
 the Willmans, a musical family
  that influenced him in every aspect during his teen-years;}
\centerline{\small\it
  and his music teachers (time-ordered)}
\centerline{\small\it
 Cheng-Yo Lin, Bai-Chung Chen, Natalia Colocci, Kathy McClure,
 Janet Maestre}
\centerline{\small\it
 who brought another dimension to his life;}
\centerline{\small\it
 and to Ann and Ling-Miao for the double flutes/flute$^{\ast}$-piano duets
 he forever cherishes.}

\vspace{36em}

\baselineskip 11pt

{\footnotesize
\noindent
$^{\ast}$(From C.-H.L.)
 Special thanks to {\it Langus},
 a then-9-year-old kid when entering accidentally my life
  and who motivated me to re-pick up the instrument
  after its being discarded for more than a decade.
} 

\end{titlepage}

\newpage
$ $

\vspace{-3em}

\centerline{\sc
 D0-Brane Realization of Resolution of Singular Curve
 } %

\vspace{2em}


\begin{flushleft}
{\Large\bf 0. Introduction and outline.}
\end{flushleft}
The work [D-M] of Michael Douglas and Gregory Moore
  on resolution of singularities of a superstring target-space $Y$
  via a D-brane probe
  (i.e., the realization of a resolution $Y^{\prime}$ of $Y$
   as a space of vacua
       -- namely, a moduli space in quantum-field-theoretical sense --
   of the world-volume quantum field theory of the D-brane probe)
 has influenced many studies
  both on the mathematics and the string-theory side.
(See also a related work [J-M] of Clifford Johnson and Robert Myers.)
The attempt to understand the underlying geometry behind
the setup
 of [D-M] is indeed part of the driving force that leads us to
 the current setting of D-branes in the project (cf.\ [L-Y1] and [L-Y2]).
Based on
 examples\footnote{Unfamiliar
                       readers are highly recommended to use keyword
                       search to get a taste of the vast literature.}
  from superstring/D-brane theory since [D-M],
 the richness and the complexity of the stack
  ${\frak M}^{0^{A\!z^f}_{\;p}}\!\!(Y)$
  of punctual D0-branes on a variety $Y$,  and
 as a guiding question,
we lay down in this note\footnote{In part,
                  for a subsection of a talk
                  under the title
                  `{\it Azumaya noncommutative geometry and D-branes
                        - an origin of the master nature of D-branes}'
                  to be delivered in the workshop
                 {\sl Noncommutative algebraic geometry and D-branes},
                  December 12 -- 16, 2011,
                  organized by Charlie Beil, Michael Douglas, and Peng Gao,
                  at Simons Center for Geometry and Physics,
                  Stony Brook University, Stony Brook, NY.
                  } %
 a conjecture
 that any resolution $Y^{\prime}\rightarrow Y$
  of a variety $Y$ over ${\Bbb C}$
 can be factored through an embedding of $Y^{\prime}$
 into the stack ${\frak M}^{0^{A\!z^f}_{\;p}}_r\!\!(Y)$
 of punctual D$0$-branes of rank $r$ on $Y$ for $r\ge r_0$ in ${\Bbb N}$,
 where $r_0$ depends on the germ of singularities of $Y$;
cf.~Sec.~1.
For the one-dimensional case,
we prove that this conjecture holds
 for the resolution $\rho: C^{\prime}\rightarrow C$
  of a reduced singular curve $C$ over ${\Bbb C}$;
cf.~Sec.~2.
In string-theoretical language, this says that
 the resolution $C^{\prime}$ of a singular curve $C$ always arises
  from an appropriate D0-brane aggregation on $C$  and that
 the rank of the Chan-Paton module of the D0-branes involved
  can be chosen to be arbitrarily large.

\bigskip

\noindent
{\it Remark 0.1.}\ {\it $[$another aspect$]$.}
 It should be noted that there is another direction
  of D-brane resolutions of singularities (e.g.\ [As], [Br], [Ch]),
  from the point of view of (hard/massive/solitonic) D-branes
  (or more precisely B-branes)
  as objects in the bounded derived category of coherent sheaves.
 Conceptually that aspect and ours (for which D-branes are soft
  in terms of string tension) are in different regimes
  of a refined Wilson's theory-space of $d=2$ supersymmetric field
  theory-with-boundary on the open-string world-sheet.\footnote{{\it For
                              mathematicians}$\,$:
                              See
                              [W-K] for the origin of the notion of
                               Wilson's theory-space and, for example,
                              [H-I-V] and [H-H-P]
                               for the case of
                               $d=2$ supersymmetric quantum field theories
                               with boundary.}
 Being so, there should be an interpolation between these two aspects.
 It would be very interesting to understand such details.

\bigskip

\noindent
{\bf Convention.}
 Standard notations, terminology, operations, facts in
  (1) algebraic geometry;
  (2) coherent sheaves;
  (3) resolution of singularities;
  (4) stacks$\,$
 can be found respectively in$\,$
  (1) [Ha];$\,$
  (2) [H-L];$\,$
  (3) [Hi], [Ko];$\,$
  (4) [L-MB].
 \begin{itemize}
  \item[$\cdot$]
   All varieties, schemes and their products are over ${\Bbb C}$;
   a `{\it curve}' means a $1$-dimensional proper scheme over ${\Bbb C}$;
   a `{\it stack}' means an {\it Artin stack}.

  \item[$\cdot$]
   The `{\it support}' $\Supp({\cal F})$
    of a coherent sheaf ${\cal F}$ on a scheme $Y$
    means the {\it scheme-theoretical support} of ${\cal F}$;
   ${\cal I}_Z$ denotes the {\it ideal sheaf} of
    a subscheme of $Z$ of a scheme $Y$.

  \item[$\cdot$]
   The current note continues the study in
    [L-Y1] (arXiv:0709.1515 [math.AG], D(1)),
    [L-Y2] (arXiv:0901.0342 [math.AG], D(3)), and
    [L-Y3] (arXiv:0907.0268 [math.AG], D(4))
    with some background from
    [L-L-S-Y] (arXiv:0809.2121 [math.AG], (2)).
   A partial review of D-branes and Azumaya noncommutative geometry
    is given in [L-Y4] (arXiv:1003.1178 [math.SG], D(6)).
   Notations and conventions follow these early works when applicable.
 \end{itemize}

\bigskip

\newpage

\begin{flushleft}
{\bf Outline.}
\end{flushleft}
{\small
\baselineskip 12pt  
\begin{itemize}
 \item[0.]
  Introduction.

 \item[1.]
  The stack of punctual D0-branes on a variety and an abundance conjecture.
  \begin{itemize}
   \item[$\cdot$]
    D-branes as morphisms from Azumaya noncommutative spaces
    with a fundamental module.

   \item[$\cdot$]
    The stack of punctual D0-branes on a variety and
    an abundance conjecture.
  \end{itemize}

 \item[2.]
  Realizations of resolution of singular curves via D0-branes.
  \vspace{-.6ex}
  \begin{itemize}
   \item[$\cdot$]
    Basic setup and a criterion for nontrivial extensions of modules.

   \item[$\cdot$]
    Separation of points in $\rho^{-1}(p)$ via punctual D0-branes at $p$.

   \item[$\cdot$]
    Construction of embeddings
     $C^{\prime}\hookrightarrow {\frak M}^{0^{A\!z^f}_{\;p}}\!\!(C)$
     that descend to $\rho$.
  \end{itemize}
  %
  %
  %
\end{itemize}
} 

\bigskip

\bigskip

\section{The stack of punctual D0-branes on a variety and
         an abundance conjecture.}

We collect a few most essential definitions and setups
 for this sub-line of the project.
Readers are referred to [L-Y4]
 for a more thorough review of the first part of the project
 and stringy-theoretical remarks on
 how inputs from [Po1], [Po2], and [Wi] lead to such a setting.

\bigskip

\begin{flushleft}
{\bf D-branes as morphisms from Azumaya noncommutative spaces\\
     with a fundamental module.}
\end{flushleft}
Our starting point is the following prototypical definition of
D-branes that comes from a mathematical understanding
 of [Po1], [Po2] from Joseph Polchinski and [Wi] from Edward Witten
 based on how Alexandre Grothendieck developed
 the theory of schemes in modern (commutative) algebraic geometry:

\begin{sdefinition}
{\bf [D-brane].} {\rm
 Let $Y$ be a variety (over ${\Bbb C}$).
 A {\it D-brane} on $Y$ is a {\it morphism} $\varphi$
 from an Azumaya noncommutative space-with-a-fundamental-module
  $(X^{A\!z},{\cal E}) := (X,{\cal O}_X^{A\!z},{\cal E})$ to $Y$.
 Here, $X$ is a scheme over ${\Bbb C}$,
       ${\cal E}$ a locally free ${\cal O}_X$-module, and
       ${\cal O}_X^{A\!z}=\Endsheaf_{{\cal O}_X}({\cal E})\,$;
  and $\varphi$ is defined through {\it an equivalence class
  of gluing systems of ring homomorphisms}
  given by $\varphi^{\sharp}:{\cal O}_Y\rightarrow {\cal O}_X^{A\!z}$.
 The rank of ${\cal E}$ is called the {\it rank} of the D-brane.
}\end{sdefinition}

Similar to the fact that the data of a morphism $f:X\rightarrow Y$
 between schemes can be encoded completely
 by its graph $\Gamma_f$ as a subscheme in $X\times Y$,
the data of $\varphi$ is also encoded completely by its graph
 $\Gamma_{\varphi}\,$:

\begin{sdefinition}
{\bf [$\varphi$ in terms of its graph $\Gamma_{\varphi}$].} {\rm
 The {\it graph} of a morphism in Definition~1.1
  is given by an ${\cal O}_{X\times Y}$-module $\tilde{\cal E}$
  that is flat over $X$ and of relative dimension $0$.
 In detail,
 let
  $\pr_1:X\times Y\rightarrow X$,
  $\pr_2:X\times Y\rightarrow Y$ be the projection map, and
  $f_{\varphi}:\Supp(\tilde{\cal E})\rightarrow Y$
   be the restriction of $\pr_2$.
 Then $\tilde{\cal E}$ defines a morphism $\varphi$ in Definition~1.1
  as follows:
  \begin{itemize}
   \item[$\cdot$]
    ${\cal E}=\pr_{1\ast}\tilde{\cal E}$;

   \item[$\cdot$]
    note that
     $\Supp(\tilde{\cal E})$ is affine over $X$;
    thus,
    the gluing system of ring homomorphisms\\
     $f_{\varphi}^{\sharp}:
      {\cal O}_Y\rightarrow {\cal O}_{Supp(\tilde{\cal E})}$
    defines a gluing system of ring-homomorphisms\\
     $\varphi^{\sharp}:{\cal O}_Y \rightarrow
             \Endsheaf_{{\cal O}_X}({\cal E})={\cal O}_X^{A\!z}$,
     which defines $\varphi$.
  \end{itemize}
}\end{sdefinition}

\noindent
It is worth emphasizing that,
 {\it unlike} the standard setting for a morphism
      between ringed topological spaces in commutative geometry,
 in general $\varphi$ specifies only a correspondence from $X$ to $Y$
 via the diagram
 $$
  \xymatrix{
     X_{\varphi}:=\Supp(\tilde{\cal E})
       \ar[rr]^-{f_{\varphi}} \ar[d]_{\pi_{\varphi}} && Y   \\
     X  && & ,
  }
 $$
 {\it not} a morphism from $X$ to $Y$.

\bigskip

Definition~1.2 suggests another equivalent description of $\varphi\,$.

\begin{sdefinition}
{\bf [$\varphi$ as morphism to stack of D0-branes].} {\rm
 Let ${\frak M}^{0^{{A\!z}^f}}(Y)$ be the stack of $0$-dimensional
  ${\cal O}_Y$-modules.
 It follows from Definition~1.2
  that this is precisely the stack of D0-branes on $Y$
   in the sense of Definition~1.1 and, hence, the notation.
 Then, a morphism $\varphi$ in Definition~1.1
  is specified by a morphism
  $\hat{\varphi}:X\rightarrow {\frak M}^{0^{{A\!z}^f}}(Y)$.
}\end{sdefinition}

\bigskip

\begin{flushleft}
{\bf The stack of punctual D0-branes on a variety
     and an abundance conjecture.}
\end{flushleft}

\begin{sdefinition}
{\bf [stack of punctual D0-branes].} {\rm
 Let $Y$ be a variety.
 By a {\it punctual} $0$-dimensional ${\cal O}_Y$-module,
 we mean a $0$-dimensional ${\cal O}_Y$-module ${\cal F}$
  whose $\Supp({\cal F})$ is a single point
  (with structure sheaf an Artin local ring).
 By Definition~1.2, ${\cal F}$ specifies a D0-brane on $Y$,
  which is called a {\it punctual D0-brane}.
 It is a morphism from an Azumaya point with a fundamental module to $Y$
  that takes the fundamental module to a punctual $0$-dimensional
  ${\cal O}_Y$-module.
 Let ${\frak M}_r^{0^{A\!z^f}_{\;p}}\!\!(Y)$
  be the {\it stack of punctual D0-branes of rank $r$ on a variety $Y$}.
 It has an Artin stack with atlas constructed from Quot-schemes.
 There is a morphism
  $\pi_Y:{\frak M}^{0^{A\!z^f}_{\;p}}\!\!(Y) \rightarrow Y$
  that takes ${\cal F}$ to $\Supp({\cal F})$
   with the reduced scheme structure.
 $\pi_Y$ is essentially the Hilbert-Chow/Quot-Chow morphism.
}\end{sdefinition}

The following two conjectures are motivated by the various examples
 in string theory concerning D-brane resolution of singularities
 of a superstring target-space  and
the richness and the complexity of the stack
 ${\frak M}^{0^{A\!z^f}_{\;p}}\!\!(Y)\,$:

\begin{sconjecture}
{\bf [D0-brane resolution of singularity].}
 Any resolution $Y^{\prime}\rightarrow Y$ of a variety $Y$
  can be factored through an embedding of $Y^{\prime}$
  into the stack ${\frak M}^{0^{A\!z^f}_{\;p}}_r\!\!(Y)$
  of punctual D$0$-branes of rank $r$ on $Y$
  for any $r\ge r_0$ in ${\Bbb N}$,
  where $r_0$ depends only on the germ of singularities of $Y$.
\end{sconjecture}

Conjecture~1.5 is a weaker form of the following stronger form
 of an abundance conjecture:

\begin{sconjecture}
{\bf [abundance].}
 Any birational morphism $Y^{\prime}\rightarrow Y$ between varieties
  over ${\Bbb C}$ can be factored through an embedding of $Y^{\prime}$
  into the stack ${\frak M}^{0^{A\!z^f}_{\;p}}_r\!\!(Y)$
  of punctual D$0$-branes of rank $r$ on $Y$
  for any $r\ge r_0$ in ${\Bbb N}$,
  where $r_0$ depends only on the germ of singularities of $Y$
   and the germ of singularities of $Y^{\prime}$.
\end{sconjecture}

This says that all the birational models of and over $Y$ are
 already contained in the stack ${\frak M}^{0^{A\!z^f}_{\;p}}\!\!(Y)$
 of punctual D0-branes on $Y$.
All the birational transitions between birational models of and over $Y$
 happens as correspondences inside ${\frak M}^{0^{A\!z^f}_{\;p}}\!\!(Y)$
   (and hence the name of the conjecture)
 -- an intrinsic stack over $Y$, locally of finite type,
    that is canonically associated to $Y$.

\begin{sremark}
{$[\,$string-theoretical remark$\,]$.} {\rm
 A standard setting (cf.\ [D-M]) in D-brane resolution of singularities
  of a (complex) variety $Y$ (which is a singular Calabi-Yau space in the context
  of string theory) is to consider
   a super-string target-space-time
    of the form ${\Bbb R}^{(9-2d)+1}\times Y$  and
   an (effective-space-time-filling) D$(9-2d)$-brane
    whose world-volume sits in the target space-time
    as a submanifold of the form ${\Bbb R}^{(9-2d)+1}\times\{p\}$.
 Here, $d$ is the complex dimension of the variety $Y$
  and $p\in Y$ is an isolated singularity of $Y$.
 When considering only the geometry of the internal part of this setting,
  one sees only a D0-brane on $Y$.
 This explains the role of D0-branes in the statement of
  Conjecture~1.5 and Conjecture~1.6.
 In the physics side, the exact dimension of the D-brane
  (rather than just the internal part) matters
 since
  supersymmetries and their superfield representations
   in different dimensions are not the same  and, hence,
  dimension does play a role in writing down
   a supersymmetric quantum-field-theory action for the world-volume
   of the D$(9-2d)$-brane probe.
 In the above mathematical abstraction, these data are now reflected
  into the richness, complexity, and a master nature of the stack
  ${\frak M}^{0^{A\!z^f}_{\;p}}_r\!\!(Y)$
   that is intrinsically associated to the internal geometry.
 The precise dimension of the D-brane as an object sitting in
  or mapped to the whole space-time becomes irrelevant.
}\end{sremark}

\bigskip

\section{Realizations of resolution of singular curves via D0-branes.}

Let
 $C$ be a reduced singular curve over ${\Bbb C}$ and
 $$
  \rho\; :\; C^{\prime}\; \longrightarrow\;  C
 $$
 be the resolution of singularities of $C$.
In the current $1$-dimensional case,
 the singularities of $C$ are isolated and
 $\rho$ is realized by the normalization of $C$.
In particular, $\rho$ is an affine morphism.
The built-in ${\cal O}_C$-module homomorphism
 $\rho^{\sharp}: {\cal O}_C\rightarrow \rho_{\ast}{\cal O}_{C^{\prime}}$
 determines a subsheaf ${\cal A}_C\subset {\cal O}_{C^{\prime}}$
 of ${\Bbb C}$-subalgebras with the induced morphism
 $C^{\prime}\rightarrow \boldSpec{\cal A}_C$ identical to $\rho$.
Let
 $p^{\prime}\in C^{\prime}$ be a closed point,
 $p:=\rho(p^{\prime})$, and
 ${\frak m}_{p^{\prime}}=(t)$ (resp.\ ${\frak m_p}$)
  be the maximal ideal of ${\cal O}_{C^{\prime},p^{\prime}}$
   (resp.\ ${\cal O}_{C,p}$).
Then
 $\rho^{\sharp}({\frak m}_p)\cdot{\cal O}_{C^{\prime},p^{\prime}}
  =(t^{n_{p^{\prime}}})$
 for some $n_{p^{\prime}}\in {\Bbb N}$.
$n_{p^{\prime}}>1$
 if and only if $p\in C_{sing}:=$ the singular locus of $C$.
We show in this section that:

\begin{sproposition}
{\bf [one-dimensional case].}
 Conjecture~1.5 holds for $\rho:C^{\prime}\rightarrow C$.
 Namely, there exists an $r_0\in {\Bbb N}$
  depending only on
   the tuple $(n_{p^{\prime}})_{\rho(p^{\prime})\in C_{sing}}$
    and a (possibly empty) set
    $\{
     \bii(p)\,:\, p\in C_{sing}\,,\,
                     \mbox{C has multiple branches at $p$}\,\}$
    (cf.\ Definition~2.6),
   both associated to the germ of $C_{sing}$ in $C$,
 such that,
 for any $r\ge r_0$,
  there exists an embedding
  $\tilde{\rho}: C^{\prime}\hookrightarrow
   {\frak M}^{0^{A\!z^f}_{\;p}}_r\!\!(C)$
  that makes the following diagram commute:
  $$
   \xymatrix{
     &&  {\frak M}^{0^{A\!z^f}_{\;p}}_r\!\!(C) \ar[d]^-{\pi_C} \\
    C^{\prime}\hspace{1ex}\ar @{^{(}->}[rru]^-{\tilde{\rho}} \ar[rr]^-{\rho}
     && \hspace{1ex}C\hspace{1ex}  &.
   }
  $$
\end{sproposition}

\bigskip

\begin{flushleft}
{\bf Basic setup and a criterion for nontrivial extensions of modules.}
\end{flushleft}
Consider the induced affine morphism
 $\id_{C^{\prime}}\times\rho:
   C^{\prime}\times C^{\prime}\rightarrow C^{\prime}\times C$.
Let
  $\pr_1^{\prime}: C^{\prime}\times C^{\prime}\rightarrow C^{\prime}$
   (the first factor),
  $\pr_2^{\prime}: C^{\prime}\times C^{\prime}\rightarrow C^{\prime}$
   (the second factor),
  $\pr_1: C^{\prime}\times C \rightarrow C^{\prime}$, and
  $\pr_2: C^{\prime}\times C \rightarrow C$
 be the projection maps.
Let $\tilde{\cal E}^{\prime}$
 be a coherent sheaf on $C^{\prime}\times C^{\prime}$
 that is flat over $C^{\prime}$ under $\pr_1^{\prime}$,
 with support in an infinitesimal neighborhood of the diagonal
 $\Delta_{C^{\prime}}\subset C^{\prime}\times C^{\prime}$.
Then
 $\tilde{\cal E}
  := (\id_{C^{\prime}}\times \rho)_{\ast}(\tilde{\cal E}^{\prime})$
 is a coherent sheaf on $C^{\prime}\times C$
 that is flat over $C^{\prime}$ under $\pr_1$,
 with support in an infinitesimal neighborhood of the graph
 $\Gamma_{\rho}$ of $\rho$ in $C^{\prime}\times C$.

\begin{slemma}
{\bf [commutativity of push-forward and restriction].}
 Let $p^{\prime}\in C^{\prime}$ be a closed point.
 Then
  $(\id_{C^{\prime}}\times \rho)_{\ast}
    (\tilde{\cal E}^{\prime}|_{\{p^{\prime}\}\times C^{\prime}})
   = \tilde{\cal E}|_{\{p^{\prime}\}\times C}$.
\end{slemma}

\begin{proof}
 As $\tilde{\cal E}^{\prime}$ is flat over $C^{\prime}$
  under $\pr_1^{\prime}$, one has the exact sequence
  $$
   0\;\longrightarrow\;
   {\cal I}_{\{p^{\prime}\}\times C^{\prime}}
    \otimes_{{\cal O}_{C^{\prime}}}\tilde{\cal E}^{\prime}\;
   \longrightarrow\;
   \tilde{\cal E}^{\prime}\; \longrightarrow\;
   \tilde{\cal E}^{\prime}|_{\{p^{\prime}\}\times C^{\prime}}\;
   \longrightarrow\; 0\,.
  $$
 Since $\id_{C^{\prime}}\times \rho$ is affine,
  $(\id_{C^{\prime}}\times \rho)_{\ast}:
   \CohCategory(C^{\prime})\rightarrow \CohCategory(C)$
  is exact and one has
  $$
   \begin{array}{cccccccccl}
    0  & \longrightarrow
       & (\id_{C^{\prime}}\times \rho)_{\ast}
          ({\cal I}_{\{p^{\prime}\}\times C^{\prime}}
          \otimes_{{\cal O}_{C^{\prime}}} \tilde{\cal E}^{\prime})
       & \longrightarrow
       & \tilde{\cal E}
       & \longrightarrow
       & (\id_{C^{\prime}}\times \rho)_{\ast}
         (\tilde{\cal E}^{\prime}|_{\{p^{\prime}\}\times C^{\prime}})
       & \longrightarrow  & 0 \\[1.2ex]
    && \|                     \\[.4ex]
    && {\cal I}_{\{p^{\prime}\}\times C}
         \otimes_{{\cal O}_{C^{\prime}}} \tilde{\cal E}
    &&&&&&&,
   \end{array}
  $$
  where the top horizontal line is an exact sequence.
 This proves the lemma.

\end{proof}

\begin{sremark-notation}
{$[$general restriction over a base$]$.} {\rm
 Lemma~2.2 holds more generally with $p^{\prime}$ replaced
  by a subscheme of $C^{\prime}$, by the same proof with the replacement.
 We'll denote the restriction of a coherent sheaf $\tilde{\cal F}^{\prime}$
  (resp.\ $\tilde{\cal F}$)
  on $C^{\prime}\times C^{\prime}$ (resp.\ $C^{\prime}\times C$)
  over a subscheme $Z^{\prime}$ of the base $C^{\prime}$ by
  $\tilde{\cal F}^{\prime}_{Z^{\prime}}$
  (resp.\ $\tilde{\cal F}_{Z^{\prime}}$).
}\end{sremark-notation}

Let $v_{p^{\prime}}\simeq \Spec({\Bbb C}[\varepsilon])$,
  where $\varepsilon^2=0$,
 be the subscheme of the base $C^{\prime}$
  that corresponds to the ${\Bbb C}$-algebra quotient
  ${\cal O}_{C^{\prime},p^{\prime}}\rightarrow {\Bbb C}[\varepsilon]$
  with $t\mapsto \varepsilon$.
Then the restriction of $\tilde{\cal E}^{\prime}$ over $v_{p^{\prime}}$
 determines an element
 $\alpha^{\prime}_{p^{\prime}}\in
  \Ext^1_{C^{\prime}}(\tilde{\cal E}^{\prime}_{p^{\prime}},
                      \tilde{\cal E}^{\prime}_{p^{\prime}})$.
Similarly, the restriction of ${\cal E}$ over $v_{p^{\prime}}$
 determines an element
  $\alpha_{p^{\prime}}=:\rho_{\ast}\alpha^{\prime}_{p^{\prime}} \in
   \Ext^1_C(\tilde{\cal E}_{p^{\prime}},\tilde{\cal E}_{p^{\prime}})$.
Let $p:=\rho(p^{\prime})$  and
recall
 $t\in {\cal O}_{C^{\prime},p^{\prime}}$ and
 $n_{p^{\prime}}\in {\Bbb N}$
 from the beginning of this section.
Let us first state an elementary criterion for non-splitability
 of a short exact sequence, whose proof is immediate:

\begin{slemma}
{\bf [criterion of non-splitability].}
 Let
  $W$ be a scheme  and
  $$
   0\; \longrightarrow\; {\cal F}_2\; \longrightarrow\; {\cal G}\;
       \longrightarrow\; {\cal F}_1\; \longrightarrow\; 0
  $$
   be an exact sequence of ${\cal O}_W$-modules
   that represents a class
   $\beta\in\Ext^1_W({\cal F}_1,{\cal F}_2)$.
 Suppose that
  there exist a point $w\in W$ and a local function $f\in {\cal O}_{W,w}$
  such that,
  for the associated ${\cal O}_{W,w}$-modules (still denoted the same),
   $f\cdot {\cal F}_1=f\cdot{\cal F}_2=0$ while $f\cdot {\cal G}\ne 0$.
 Then, $\beta\ne0$; namely, the above sequence doesn't split.
\end{slemma}

\begin{scorollary}
{\bf [push-forward of jet].}
 Continuing the main-line discussions and notations.
 Let $\alpha^{\prime}$ be given by the exact sequence
  $$
   0\; \longrightarrow\;
   \tilde{\cal E}^{\prime}_{p^{\prime}}\; \longrightarrow\;
   \tilde{\cal F}^{\prime}\; \stackrel{j}{\longrightarrow}\;
   \tilde{\cal E}^{\prime}_{p^{\prime}}\; \longrightarrow\; 0
  $$
  of ${\cal O}_{C^{\prime}}$-modules.
 Denote the same for the associated exact sequence of
  ${\cal O}_{C^{\prime},p^{\prime}}$-modules.
 As such, suppose that
  there is an $l\in {\Bbb N}$ such that
  $(t^{n_{p^{\prime}}})^l\cdot \tilde{\cal E}^{\prime}=0$ while
  $(t^{n_{p^{\prime}}})^{l+1}\cdot \tilde{\cal F}^{\prime}\ne 0$.
 Then
  $\alpha\ne 0$ in
  $\Ext^1_C(\tilde{\cal E}_{p^{\prime}},\tilde{\cal E}_{p^{\prime}})$.
\end{scorollary}

\begin{proof}
 Note that the multiplication of $t$ by an invertible element in
   ${\cal O}_{C^{\prime},p^{\prime}}$
  (i.e.\ by an element in
   ${\cal O}_{C^{\prime},p^{\prime}}-{\frak m}_{p^{\prime}}$)
 won't alter its nilpotency behavior on the modules in question.
 The corollary follows immediately from Lemma~2.4
  and the observation that,
  up to a multiplication by an invertible element in
    ${\cal O}_{C^{\prime},p^{\prime}}$,
  one may assume that
   $t^{n_{p^{\prime}}}\in \rho^{\sharp}({\cal O}_{C,p})$.

\end{proof}

\bigskip

\begin{flushleft}
{\bf Separation of points in $\rho^{-1}(p)$ via punctual D0-branes at $p$.}
\end{flushleft}
Let $p\in C_{sing}$ and $\hat{C}$ be the formal neighborhood
 (as an ind-scheme) of $p$ in $C$.
Then
 each irreducible component $\hat{C}_i$, $i=1\,,\,\cdots\,,\,k$,
 of $\hat{C}$ corresponds to a branch of the germ of $p$ in $C$.
Assume that $k\ge 2$.
Then the intersection of two distinct components
 $\hat{C}_i$ and $\hat{C}_j$ of $\hat{C}$
 is represented by a punctual $0$-dimensional subscheme
 $Z_{ij}=Z_{ji}$ of $C$ at $p$ of finite length $l_{ij}=l_{ji}$.

\begin{sdefinition}
{\bf [branch intersection index]}. {\rm
 For $k\ge 2$,
 define the {\it branch intersection index} $\bii(p)$ at $p\in C_{sing}$
  to be
  $$
    \bii(p)\; :=\; \max\{l_{ij}\::\:1\le i,j\le k; i\ne j\}\,.
  $$
}\end{sdefinition}

Let
 $p\in C_{sing}$,
 $\rho^{-1}(p)=\{p^{\prime}_1,\,\cdots\,p^{\prime}_k\}$, and
 $\hat{C}^{\prime}_i$
  be the formal neighborhood of $p^{\prime}_i$ in $C^{\prime}$.
Then $\rho: C^{\prime}\rightarrow C$ induces a morphism
 $\hat{\rho}_i:\hat{C}^{\prime}_i\rightarrow \hat{C}$
 of ind-schemes, for $i=1,\,\ldots,k$.
The image $\hat{\rho}_i(\hat{C}^{\prime}_i)$ is a branch of $\hat{C}$,
 which we may assume to be $\hat{C}_i$, after relabeling,
 since different $\hat{C}^{\prime}_i$'s
 are mapped to different branches of $\hat{C}$ under $\hat{\rho}_i$.
Let
 ${\frak m}_{p^{\prime}_i}=(u_i)$
  be the maximal ideal of ${\cal O}_{C^{\prime},\,p^{\prime}_i}$;
\begin{itemize}
 \item[$\cdot$]
  ${\cal F}^{\prime}_{i;l}$
   be the $0$-dimensional ${\cal O}_{C^{\prime}}$-module
   ${\cal O}_{C^{\prime},\,p^{\prime}_i}/(u_i^{n_{p^{\prime}_i}\,l})\,$;

 \item[$\cdot$]
  $\hat{\cal F}^{\prime}_{i;l}$
   be the ${\cal O}_{\hat{C}^{\prime}_i}$-module
   associated to ${\cal F}^{\prime}_{i;l}\,$;

 \item[$\cdot$]
  ${\cal F}_{i;l}$
   be the ${\cal O}_C$-module $\rho_{\ast}{\cal F}^{\prime}_{i;l}\,$;

 \item[$\cdot$]
  $\hat{\cal F}_{i;l}$
   be the ${\cal O}_{\hat{C}}$-module
   $\hat{\rho}_{i\,\ast}\hat{\cal F}^{\prime}_{i;l}
    = \widehat{\rho_{\ast}{\cal F}^{\prime}_{i;l}}\,$.
\end{itemize}
Then, one has the following lemma:

\begin{slemma}
{\bf [separation by punctual modules].}
 $\length(\Supp({\cal F}_{i;l}))\ge l$ and
 $\Supp(\hat{\cal F}_{i;l})\subset \hat{C}_i$.
 In particular,
 if $l>\bii(p)$,
  then ${\cal F}_{1;l}\,,\,\cdots\,,\,{\cal F}_{k;l}$
   are punctual $0$-dimensional ${\cal O}_C$-modules at $p$
   that are non-isomorphic to each other.
\end{slemma}

\begin{proof}
 As in the previous theme, we may assume that
  $u_i^{n_{p^{\prime}_i}}=\rho^{\sharp}(f_i)$
   for some $f_i\in {\frak m}_p\subset {\cal O}_{C,p}$.
 Let $h\in {\Bbb C}[x]$ be a polynomial in one variable.
 Then, by construction,
  $h(u_i^{n_{p^{\prime}_i}})\cdot {\cal F}^{\prime}_{i;l}=0$
  if and only if $h\in (x^l)$.
 In other words, $h(f_i)\cdot {\cal F}_{i;l}=0$
  if and only of $h\in (x^l)$.
 It follows that
  there exists a local section $m_{i;l}$ of ${\cal F}_{i;l}$
  such that $f_i^{l-1}\cdot m_{i;l}\ne 0$.
 Consider the sub-${\cal O}_C$-module
  ${\cal O}_C\cdot m_{i;l}\simeq {\cal O}_C/\Ann(m_{i;l})$
   of ${\cal F}_{i;l}$,
  where $\Ann(m_{i;l})$ is the annihilator of $m_{i;l}$
   in ${\cal O}_{C,p}$.
 Then,
  $$
   m_{i;l}\,,\;  f_i\cdot m_{i;l}\,,\; \cdots\,,\;
   f_i^{l-1}\cdot m_{i;l}
  $$
  are ${\Bbb C}$-linearly independent in ${\cal F}_{i;l}\,$,
  which implies that
  $$
   1\,,\; f_i\,,\;\cdots\,,\;f_i^{l-1}
  $$
  are ${\Bbb C}$-linearly independent in ${\cal O}_{C,p}$.
 Since
  $$
   \Span_{\Bbb C}\{1\,,\, f_i\,,\,\cdots\,,\,f_i^{l-1}\}
    \cap \Ann(m_{i;l})\; =\; 0
  $$
  as ${\Bbb C}$-vector subspaces in ${\cal O}_{C,p}$,
 one has that
  $\length(\Supp({\cal O}_{C,p}/\Ann(m_i)))\ge l$ and, hence,
  that $\length(\Supp({\cal F}_{i;l}))\ge l$.
 The rest of the lemma are immediate.

\end{proof}

\noindent
We say that {\it
 $p^{\prime}_1\,,\,\cdots\,,\,p^{\prime}_k
   \in \rho^{-1}(p) \subset C^{\prime}$
 are separated by the punctual ${\cal O}_C$-modules
  ${\cal F}_{1;l}\,,\,\cdots\,,\,{\cal F}_{k;l}$ at $p\in C$}
when ${\cal F}_{1;l}\,,\,\cdots\,,\,{\cal F}_{k;l}$
 as constructed above are non-isomorphic to each other.

\bigskip

\begin{flushleft}
{\bf Construction of embeddings
     $C^{\prime}\hookrightarrow {\frak M}^{0^{A\!z^f}_{\;p}}\!\!(C)$
     that descend to $\rho$.}
\end{flushleft}
We now proceed to prove Proposition~2.1 in three steps.

\bigskip

\noindent
{\it Step $(a)\,$}:
{\it Examination of a local model.}

\medskip

\noindent
Consider the local ring
 $\,{\cal O}_{C^{\prime}\times C^{\prime},\,(p^{\prime},\,p^{\prime})}
  = {\cal O}_{C^{\prime},\,p^{\prime}}
       \otimes_{\Bbb C}{\cal O}_{C^{\prime},\,p^{\prime}}\,$.
(For simplicity of phrasing, here we use `$=$' to mean
 `standard canonical isomorphism'.)
Let
 ${\frak m}_{p^{\prime}}=(t_1)\subset {\cal O}_{C^{\prime},\,p^{\prime}}$
  be the maximal ideal of the first factor  and
 ${\frak m}_{p^{\prime}}=(t_2)\subset {\cal O}_{C^{\prime},\,p^{\prime}}$
  be the maximal ideal of the second factor.
Given $r\in {\Bbb N}$,
compare the following two quotient
 ${\cal O}_{C^{\prime}\times C^{\prime},\,
            (p^{\prime},\,p^{\prime})}$-modules:
 \begin{eqnarray*}
  M_1\;:=\;
   \frac{ {\cal O}_{C^{\prime},\,p^{\prime}}
             \otimes_{\Bbb C}{\cal O}_{C^{\prime},\,p^{\prime}} }
       { \left((t_1\otimes1-1\otimes t_2)^r\,,\, t_1\otimes 1\right) }
   & \hspace{2em}\mbox{and}\hspace{2em}
   & M_2\;:=\:
     \frac{{\cal O}_{C^{\prime},\,p^{\prime}}
             \otimes_{\Bbb C}{\cal O}_{C^{\prime},\,p^{\prime}} }
       { \left((t_1\otimes1-1\otimes t_2)^r\,,\, t_1^2\otimes 1\right) }\,.
 \end{eqnarray*}
$M_1$ corresponds to the restriction of
 the ${\cal O}_{C^{\prime},\,p^{\prime}}
      \otimes_{\Bbb C}{\cal O}_{C^{\prime},\,p^{\prime}}$-module
 ${\cal O}_{C^{\prime},\,p^{\prime}}
              \otimes_{\Bbb C}{\cal O}_{C^{\prime},\,p^{\prime}}
     /(t_1\otimes1-1\otimes t_2)^r$,
 which is flat over $C^{\prime}$ (the first factor),
 to over $p^{\prime}\in C^{\prime}$ (the first factor)
while
$M_2$ corresponds to the restriction of
 the same ${\cal O}_{C^{\prime},\,p^{\prime}}
      \otimes_{\Bbb C}{\cal O}_{C^{\prime},\,p^{\prime}}$-module
 to over
  $v_{p^{\prime}}
   \simeq\Spec({\Bbb C}[t_1]/(t_1^2))
   \simeq \Spec({\Bbb C}[\varepsilon])
                            \subset C^{\prime}$ (the first factor).
They fit into an exact sequence,
 representing a class in $\Ext^1_{C^{\prime}}(M_1,M_1)$
 (here $C^{\prime}=$ the second factor),
 $$
  0\, \longrightarrow\; M_1\;
      \stackrel{a}{\longrightarrow}\; M_2\;
      \stackrel{b}{\longrightarrow}\; M_1\;
      \longrightarrow\; 0
 $$
 of ${\Bbb C}[\varepsilon]$-modules
 with
 $$
  \begin{array}{lcl}
   M_1 & =
       & \Span_{\Bbb C}\left\{
          1\otimes 1\,,\,1\otimes t_2^2\,,\,\cdots\,,\,
                         1\otimes t_2^{r-1}  \right\}\,; \\[1.2ex]
   M_2 & =
    & \Span_{{\Bbb C}[\varepsilon]} \left\{
       1\otimes 1\,,\,1\otimes t_2^2\,,\,\cdots\,,\,1\otimes t_2^{r-1}
                                    \right\}             \\[1.2ex]
    & =
    & \Span_{\Bbb C}\left\{
       1\otimes 1\,,\,1\otimes t_2^2\,,\,\cdots\,,\,
                        1\otimes t_2^{r-1}\,,\,
          \varepsilon\otimes 1\,,\,\varepsilon\otimes t_2^2\,,\,\cdots\,,\,
          \varepsilon\otimes t_2^{r-1} \right\}\,,
  \end{array}
 $$
 where
  $a=$ multiplication by $\varepsilon$, and
  $b=$ quotient by $\varepsilon M_1$.
As ${\Bbb C}[\varepsilon]$-modules and with respect to the above bases
 (and with a vector identified as a column vector),
 $$
  \mbox{$t_2\,$ on $\,M_1$}\; =\;
  \left[
   \begin{array}{ccccc}
    0  &                        \\
    1  & 0                      \\
       & 1 & \ddots             \\
       &   & \ddots  & 0        \\
       &   &         & 1 & 0
   \end{array}
  \right]_{r\times r}
  \hspace{2em}\mbox{and}\hspace{2em}
  \mbox{$t_2\,$ on $\,M_2$}\; =\;
  \left[
   \begin{array}{ccccc}
    0  &                        \\
    1  & 0                      \\
       & 1 & \ddots             \\
       &   & \ddots  & 0        \\
       &   &         & 1 & r\varepsilon
   \end{array}
  \right]_{r\times r}\,.
 $$
Here all the missing entries in the $r\times r$-matrices are $0$.
It follows that,
 as ${\cal O}_{C^{\prime},p^{\prime}}$ (the second factor) -modules,
 \begin{itemize}
  \item[$\cdot$]
  {\it $t_2^l\cdot M_1 = 0\;$ if and only if $\;l\ge r\;\;\;$
       while $\;\;\;t_2^l\cdot M_2 = 0\;$ if and only if $\;l\ge r+1\,$}.
 \end{itemize}
 In particular, the above short exact sequence
  (of ${\cal O}_{C^{\prime},p^{\prime}}$-modules) doesn't split.

\bigskip

\noindent
{\it Step $(b)\,$}:
{\it Construction of a local embedding
 $C^{\prime}\rightarrow
   {\frak M}^{0^{A\!z^f}_{\;p}}_{r_0}\!\!(C)$
 that descend to $\rho$, for some $r_0\in{\Bbb N}$.}

\medskip

\noindent
Let
 $$
  l_0\; :=\;
    1\, +\,\max\{\, \bii(p)\,:\, p\in C_{sing}\,,\,
                        \mbox{C has multiple branches at $p$}\,\}
 $$
  (by convention,
  $l_0=1$ if $C$ has only single branch at each $p\in C_{sing}$)$\,$
  and
 $$
  r_0\; :=\;  l_0\cdot\lcm\{n_{p^{\prime}}:p^{\prime}\in C^{\prime}\}\;
        \in\; {\Bbb N}\,.
 $$
 (Here, $\lcm =$ the `{\it least common multiple}' in ${\Bbb N}$.)
Since $n_{p^{\prime}}=1$ except for $\rho(p^{\prime})$
 in the finite set $C_{sing}$, $r_0$ is well-defined.
Furthermore, since
  $\{n_{p^{\prime}}\}_{\rho(p^{\prime})\in C_{sing}}$  and
  $\{\bii(p):p\in C_{sing}\}$ (possibly empty)
 depend only on the germ of $C_{sing}$ in $C$,
 $r_0$ depends only on the germ of $C_{sing}$ in $C$.
Let $\tilde{\cal E}^{\prime}$
 be the ${\cal O}_{C^{\prime}\times C^{\prime}}$-module
 $$
  \tilde{\cal E}^{\prime}\;
   =\; {\cal O}_{C^{\prime}\times C^{\prime}}
            /{\cal I}_{\Delta_{C^{\prime}}}^{\;r_0}
 $$
  and
 $\tilde{\cal E}
  :=(\id_{C^{\prime}}\times\rho)_{\ast}(\tilde{\cal E}^{\prime})$
  on $C^{\prime}\times C$.
Then,
 it follows from the construction and Lemma~2.7
 that the induced morphism
 $$
  \tilde{\rho}_0\; :\; C^{\prime}\; \longrightarrow\;
                       {\frak M}^{0^{A\!z^f}_{\;p}}_{r_0}\!\!(C)
 $$
 descends to $\rho$  and
 sends distinct closed points of $C^{\prime}$
  to distinct geometric points on ${\frak M}^{0^{A\!z^f}_p}_{r_0}(C)$
 (i.e.\ $\tilde{\rho}$ separates points of $C^{\prime}$).
Furthermore,
 it follow from the local study in Step (a) and Corollary~2.5
 that
 all the extension classes
  $\alpha_{p^{\prime}}\in
   \Ext^1_C(\tilde{\cal E}_{p^{\prime}},\tilde{\cal E}_{p^{\prime}})$,
  $p^{\prime}\in C^{\prime}$,
  $\tilde{\cal E}$ specifies are non-zero.
This shows that
 $\tilde{\rho}_0$ separates also tangents of $C^{\prime}$
 and hence is an embedding.

\bigskip

\noindent
{\it Step $(c)\,$}:
{\it Embeddings
     $C^{\prime}\hookrightarrow {\frak M}^{0^{A\!z^f}_{\;p}}_r\!\!(C)$
     that descend to $\rho$, for all $r> r_0$.}

\medskip

\noindent
Finally,
to obtain an embedding
 $\tilde{\rho}: C^{\prime} \rightarrow
                {\frak M}^{0^{A\!z^f}_{\;p}}_r\!\!(C)$ for $r>r_0$
 that descends to $\rho$,
observe that
 the ${\cal O}_{C^{\prime}\times C}$-module ${\cal O}_{\Gamma_{\rho}}$
 has the following properties:
 \begin{itemize}
  \item[$\cdot$]
   The corresponding extension class
   $\bar{\alpha}_{p^{\prime}}$ in $\Ext^1_C({\cal O}_p,{\cal O}_p)$,
   where $p:=\rho(p^{\prime})$, vanishes if and only if $p\in C_{sing}$.
 \end{itemize}
This implies that
 all the extension classes
  $\hat{\alpha}_{p^{\prime}}
   \in \Ext^1_C(\hat{\cal E}_{p^{\prime}},\hat{\cal E}_{p^{\prime}})$,
  $p^{\prime}\in C^{\prime}$,
  as specified by the direct sum
  $$
   \hat{\cal E}\;
    :=\; \tilde{\cal E}\oplus {\cal O}_{\Gamma_{\rho}}^{\;\oplus(r-r_0)}
  $$
  of ${\cal O}_{C^{\prime}\times C}$-modules,
  remain non-zero.
Furthermore,
 $$
  \Supp((\tilde{\cal E}
          \oplus {\cal O}_{\Gamma_{\rho}}^{\;\oplus(r-r_0)})
                                          |_{p^{\prime}\times C})\;
  =\; \Supp(\tilde{\cal E}_{p^{\prime}})
  \hspace{2em}\mbox{for all $\;p^{\prime}\in C^{\prime}$}\,.
 $$
It follows that
 the morphism
 $\tilde{\rho}: C^{\prime}\rightarrow
                {\frak M}^{0^{A\!z^f}_{\;p}}_r\!\!(C)$
 specified by $\hat{\cal E}$ on $C^{\prime}\times C$
 separates both points and tangents of $C^{\prime}$ and, hence,
 is an embedding that descend to $\rho$.

\bigskip

This concludes the proof of Proposition~2.1.

\begin{sremark} $[$non-uniqueness$\,]$. {\rm
 In general there can be other embeddings of $C^{\prime}$
  into ${\frak M}^{0^{A\!z^f}_{\;p}}_r\!\!(C)$
  that descend also to $\rho$.
 Hence, the one constructed in the proof above is by no means unique.
}\end{sremark}

\newpage
\baselineskip 13pt

\end{document}